\documentclass[11pt,a4paper,english]{article}
\usepackage[utf8]{inputenc}
\usepackage[T1]{fontenc}
\usepackage{babel}
\usepackage{amsmath}
\usepackage{mathtools}
\usepackage{amssymb}
\usepackage{amsthm}
\usepackage[all,cmtip]{xy}
\usepackage{graphicx}
\usepackage{xy}
\usepackage{geometry}
\usepackage{verbatim}
\usepackage{hyperref}
\usepackage{url}
\usepackage{extarrows}

\makeatletter
\def\@fnsymbol#1{\ensuremath{\ifcase#1\or \star\or\dagger\or \ddagger\or
   \mathsection\or \mathparagraph\or \|\or **\or \dagger\dagger
   \or \ddagger\ddagger \else\@ctrerr\fi}}
\makeatother

\usepackage{csquotes}

\usepackage{bbm}

\title{On the First Fundamental Theorem \\
for $\GL_2(K)$ and $\SL_2(K)$}
\author{Hana Mel\'{a}nov\'{a}\footnote{Supported by Austrian Academy of Sciences, ÖAW, Doc Stipendium FA506081} \ and 
Sergey Yurkevich\footnote{Supported by Austrian Science Fund, FWF, Project AP31338}}

\DeclareMathOperator{\GL}{GL}
\DeclareMathOperator{\SL}{SL}

\newcommand{\NN}[0]{\mathbb{N}}

\newcommand{\CC}[0]{\mathbb{C}}

\newcommand{\xx}[0]{\mathbbm{x}}
\newcommand{\yy}[0]{\mathbbm{y}}

\begin{document}

\newtheorem{theorem}{Theorem}
\newtheorem*{Ftheorem}{First Fundamental Theorem for $\SL_2(K)$}
\newtheorem*{Gtheorem}{First Fundamental Theorem for $\GL_2(K)$}
\newtheorem{definition}[theorem]{Definition}
\newtheorem{lemma}[theorem]{Lemma}
\newtheorem{proposition}[theorem]{Proposition}
\newtheorem{corollary}[theorem]{Corollary}
\newtheorem{example}[theorem]{Example}
\newtheorem{remark}[theorem]{Remark}

\newpage
\maketitle

\begin{abstract}
    The First Fundamental Theorem of Invariant Theory describes a minimal generating set of the invariant polynomial ring under the action of some group $G$. In this note we give an elementary and direct proof for the $\GL_2(K)$ and $\SL_2(K)$ for any infinite field $K$. 
    %This case is of particular importance for the moduli space of $n$ points of the projective line \cite{HoMiSnVa09}.
    Our proof can be generalized to $\GL_m(K)$ and $\SL_m(K)$ for $m>2$. Moreover, we present a family of counter-examples to the statements of the First Fundamental Theorems for all finite fields and $m=2$.
\end{abstract}

\noindent Consider the polynomial ring of pairs of variables $K[\xx,\yy]\coloneqq K[x_1,y_1,\dots,x_n,y_n]$ over an infinite field $K$ and its quotient field $K(\xx,\yy)$. The general linear group $\GL_2(K)$ acts on $K[\xx,\yy]$, and thus also on $K(\xx,\yy)$, from the right by the usual matrix-vector multiplication on the pairs of variables $(x_i,y_i)$. We  denote the corresponding invariant field by $K(\xx,\yy)^{\GL_2}$. Clearly, no \textit{polynomial}, except for the constant ones, can be invariant under the action of $\GL_2(K)$, i.e.,  $K[\xx,\yy]^{\GL_2}=K.$ However, it is easy to construct polynomials that are invariant under the action of $\SL_2(K)$. Let us denote the corresponding invariant ring by $K[\xx,\yy]^{\SL_2}$.  We set 
\[
f_{i,j} \coloneqq x_iy_j - y_ix_j \in K[\xx,\yy], \quad 1 \leq i,j \leq n,
\]
and denote by $K[f_{ij}]$ the polynomial ring generated over $K$ by all $f_{i,j},1\leq i,j\leq n$ and by $K(f_{ij})$ its quotient field. Notice that these polynomials satisfy the straightforward equalities
\begin{align}\label{1}
f_{i,i}=0,\quad f_{j,i}=-f_{i,j},
\end{align}
and also the \emph{Plücker relation}
\begin{align}\label{2}
f_{i,j}f_{k,l}=f_{i,k}f_{j,l}-f_{i,l}f_{j,k},
\end{align}
for all $1\leq i,j,k,l\leq n.$ Moreover, it is obvious that each $f_{i,j}$ is \textit{semi-invariant} under the action of $\GL_2(K)$, i.e.,
\[
G\cdot f_{i,j}=\det(G) f_{i,j},
\]
for any $G\in\GL_2(K)$, and hence invariant under the action of $\SL_2(K)$, justifying the inclusions $K[\xx,\yy]^{\SL_2} \supseteq K[f_{ij}]$ and $K(\xx,\yy)^{\GL_2}\supseteq K\left(\frac{f_{i,j}}{f_{k,l}}:1\leq i,j,k,l\leq n\right)$. 

The goal of this note is to give a new, elementary and self-contained proof of the First Fundamental Theorems for $\SL_2(K)$ and $\GL_2(K)$:

\begin{Gtheorem}
An element $q\in K(\xx,\yy)$ is invariant under the action of $\GL_2(K)$ if and only if $q$ can be written as a rational function in $f_{i,j}/f_{k,l},1\leq i,j,k,l\leq n$, i.e., 
\[
K(\xx,\yy)^{\GL_2}=K\left(\frac{f_{i,j}}{f_{k,l}}:1\leq i,j,k,l\leq n\right).
\]
Moreover, an invariant $q$ admits the following representation in the generators $f_{i,j}/f_{k,l}$'s:
\[
q(x_i,y_i)=q\left(\frac{f_{1,i}}{f_{1,2}},\frac{f_{2,i}}{f_{1,2}}\right).
\]
\end{Gtheorem}

\begin{Ftheorem}
An element $p\in K[\xx,\yy]$ is invariant under the action of $\SL_2(K)$ if and only if $p$ can be written as a polynomial in $f_{i,j},1 \leq i,j \leq n$, i.e., $$K[\xx,\yy]^{\SL_2}=K[f_{ij}].$$
\end{Ftheorem}

\noindent\textbf{Historical information:}
The history of the First Fundamental Theorem (shortly FFT) is long and complex. Depending on the source, it is first attributed to Clebsch \cite{Clebsch1870}, Weyl \cite{Weyl39}, Hodge \cite{Ho43} or Igusa \cite{Ig54}. We shall explain the contribution of these authors as well as provide insight into more recent approaches.

Indeed, Richman \cite[p.44]{Richman89} recognized the oldest reference \cite[p.51]{Clebsch1870} from 1870, in which Clebsch considered $\GL_2(K)$ semi-invariant polynomials by working with the so-called \textit{Aronhold operator}. Note that this proof works for all fields $K$ of characteristic $0$. Some 30 years later, Grace and Young found an easier proof of Clebsch's theorem using the \textit{Cayley $\Omega$-operator} and compared it with the original one \cite[\S28, \S35]{GrYo04}. The first reference for the $\GL_n(K)$ semi-invariant polynomials for any $n \in \NN$ is \cite[\S 5]{Turnbull60} in which Turnbull generalized the preceding ideas of the $\Omega$-operator. Then Weyl made a first breakthrough in this area by employing \textit{Capelli identities} in \cite[Thm.~2.6.A]{Weyl39}. 

Soon after in 1953, Igusa \cite[Thm.~4]{Ig54} proved the FFT for $\SL_n(K)$, where $K$ is any \textit{universal domain}, by placing it in a completely different, geometric, setting. Embedding the invariant ring into the coordinate ring of a Grassmann variety and using tools from abstract algebraic geometry, he was the first one who showed the theorem for any algebraically closed field $K$.

The next major change in perspective was done by Doubilet, Rota and Stein in \cite[p.~200-202]{DoRoSt74} where the authors first introduced the combinatorial \textit{straightening lemma} (see below) and \textit{double tableaux}, and then proved the FFT for $\GL_n(K)$ and $\SL_n(K)$ in a different but equivalent setting for arbitrary infinite fields. It shall be noted that the straightening lemma (while not named like this) was already proven by Hodge \cite[p.~27]{Ho43}, who attributed ideas to Young \cite[Thm.~1]{Young1928}. Two years later, De Concini and Procesi noted that the paper \cite{DoRoSt74} had a gap and fixed it \cite[Thm.~1.2]{CoPr76}. 
%We also mention the paper \cite{CoEiPr80} in which the previous authors and Eisenbud provided more geometric insight in the approach. 
A decade later, Barnabei and Brini \cite{BaBr86} published an article with a more elementary proof, again for infinite fields, where they managed to avoid double tableaux. 

An even more recent and new proof for all infinite fields was found by Richman \cite[\S 3]{Richman89}, in which the author described a reduction to the case $n=2$. Then Richman showed that polynomial invariants under the action of $\SL_2(K)$ are equal to the ones under the action of the special upper triangular matrices. It turned out that the latter can be described easier.  

Moreover, we mention the paper \cite{StWh89} in which Sturmfels and White presented the straightening algorithm using reduction modulo Gröbner bases. In his book \cite[\S3.2]{Sturmfels08}, Sturmfels explained how this algorithm can be used to show the FFT for $\SL_n(\CC)$. Moreover, it turns out that the direct straightening algorithm approach is rather slow for practical computations and so in the same book the author gave another more efficient algorithm for the representation of $\SL_n(K)$ invariants in terms of the generators of the invariant ring.

Finally, a more recent proof is due to Kraft and Procesi \cite[\S 8]{KrPr96}, in which they deduce the FFT for $\SL_n(K)$ and $\GL_n(K)$ for infinite fields from a generalization of Weyl's Theorems. We also refer to \cite{Dolgachev03} for extended bibliographic notes and a well-written proof using Cayley's $\Omega$-operator.\\
%\cite{Ke93}

Let us start with the proof of the First Fundamental Theorem for $\GL_2(K)$:
\begin{proof}[Proof of the First Fundamental Theorem for $\GL_2(K)$]
Let $q\in K(\xx,\yy)$ be invariant under the action of $\GL_2(K)$. Then for any matrix
\[
G=\frac{1}{ad-bc}\begin{pmatrix}
a & b\\
c & d
\end{pmatrix}
\]
satisfying $ad-bc\neq 0$ we have the equality
\begin{align}\label{gl_eq}
q(x_i,y_i)=G\cdot q(x_i,y_i)=q\left(\frac{ax_i+by_i}{ad-bc},\frac{cx_i+dy_i}{ad-bc}\right).
\end{align}
Now, the equality $(\ref{gl_eq})$ holds on an Zariski open subset of $K^{2n+4}$, namely on all tuples $(x_1,y_1,\dots,x_n,y_n,a,b,c,d)$ for which $ad-bc\neq 0$ and the denominator of $q$ does not vanish. Moreover, $|K|=\infty$, and it follows by Weyl's principle that we can consider $a,b,c,d$ to be variables. Let us now substitute $a=-y_1,b=x_1,c=-y_2,d=x_2$ into $(\ref{gl_eq})$ and obtain:
\[
q(x_i,y_i)=q\left(\frac{f_{1,i}}{f_{1,2}},\frac{f_{2,i}}{f_{1,2}}\right).
\]
This proves that $K(\xx,\yy)^{\GL_2} \subseteq K(\frac{f_{i,j}}{f_{k,l}})$. The other inclusion is clear from our considerations before.
\end{proof}

\noindent\textbf{Idea of proof for the First Fundamental Theorem for $\SL_2(K)$:}
We will reprove and use Hodge's straightening lemma \cite{Ho43} and draw inspiration from De Concini and Procesi \cite{CoPr76}. Given an invariant polynomial $p\in K[\xx,\yy]^{\SL_2}$, one may assume that $p$ is homogeneous in the variables $x_1,\dots,x_n$, let us say of degree $m$. We then let a suitably chosen $\GL_2(K)$-matrix act on $p$ in order to show that $f_{1,2}^m\cdot p \in K[f_{ij}]$. However, from this one cannot immediately conclude that $p \in K[f_{ij}]$. The problem here is the fact that relations between the elements of this ring exist. Hence, $p$ does not admit a unique representation as a polynomial in the generators $f_{i,j}$. Therefore, we investigate first the ring $K[f_{ij}]$, study its structure, when considered as a $K$-algebra, and construct a $K$-basis. Only then we can eliminate possible relations and deduce that $p \in K[f_{ij}]$.\\

Let us now start with the investigation of the ring $K[f_{ij}]$. Any product of the form $f_{i_1,j_1}f_{i_2,j_2}\cdots f_{i_m,j_m}$ can be associated with the following diagram
\begin{align*}\begin{bmatrix}
i_1 & i_2 & \cdots & i_m\\
j_1 & j_2 & \cdots & j_m
\end{bmatrix}
\end{align*}
where $i_{k}< j_{k}$ for all $k$ (using relations (\ref{1})). If we can permute the columns of the diagram in such a way that $i_1\leq i_2\leq \dots \leq i_m$ and  $j_1\leq j_2\leq \dots\leq j_m$, the diagram becomes a standard Young tableau and the corresponding product $f_{i_1,j_1}f_{i_2,j_2}\cdots f_{i_m,j_m}$ is called a \emph{standard product}. Notice, that each product $f_{i_1,j_1}f_{i_2,j_2}\cdots f_{i_m,j_m}$ can be transformed into a sum of standard product just by applying iteratively the Plücker relation.

\begin{remark}\normalfont
If a product $f_{i_1,j_1}f_{i_2,j_2}\cdots f_{i_m,j_m}$ is divisible by some $f_{i,j}$, then after applying the Plücker relation (\ref{2}) to it we get two summands which are both divisible by $f_{i,r}$ for some $r\in\{1,\dots,n\}$ and $f_{j,s}$ for some $s\in\{1,\dots,n\}$. This is due to the fact that on both sides of the Plücker relation (\ref{2}) the set of variables appearing in each summand is the same.
\end{remark}

\begin{lemma}[Straightening lemma]
The monic standard products form a $K$-basis of $K[f_{ij}]$.
\end{lemma}

This lemma was first proven by Hodge \cite{Ho43}, who attributes the idea to consider standard tableaux to Young \cite{Young1928}. However, Hodge's proof is lengthy and so we will provide a simpler argument. 

\begin{proof}
It follows from the Plücker relations that the monic standard products form a generating system. Consider now the monomial ordering on $\mathbb{N}^{2n}$ given by $x_1 \prec y_1 \prec x_2 \prec \cdots \prec y_n$. In this way, different standard products have different leading monomials which proves their linear independence.
\end{proof}

Hence, any $p \in K[f_{ij}]$ can be uniquely written as $p=\sum_{\alpha\in I} c_{\alpha}F_{\alpha}$, for some index set $I \subseteq (\{1,\dots,n\}\times\{1,\dots,n\})^N$ where $N \in \NN$, $c_\alpha \in K$ and the $F_{\alpha}$'s are standard products. 

\begin{lemma}\label{5}
Let $p=\sum_{\alpha\in I} c_{\alpha}F_{\alpha}$
be a $K$-linear combination of standard products $F_{\alpha}$. If for some $i$, the polynomial $p$ vanishes after the substitution $(x_i,y_i)=(0,0)$, i.e.,
\[
p|_{(x_i,y_i)=(0,0)}= 0,
\]
then each summand $F_{\alpha}$ is divisible by $f_{i,r_\alpha}$ for some $r_\alpha \in\{1,\dots,n\}.$
\end{lemma}

\begin{proof}
Let us assume by contradiction that there are standard products $F_{\alpha_1},\dots,F_{\alpha_k},\alpha_j\in I$ which are not divisible by any $f_{i,r},r\in\{1,\dots,n\}.$ Notice that these standard products satisfy $F_{\alpha_j}=F_{\alpha_j}|_{(x_i,y_i)=(0,0)}$ for each $j$. Then evaluating $p$ at $(x_i,y_i)=(0,0)$ gives 
\[
0=\sum_{i=1}^k c_{\alpha_i}F_{\alpha_i},
\]
which contradicts the linear independence of standard products.
\end{proof}

\begin{lemma}\label{lem:qf12}
Let $q\in K[\xx,\yy]$ be a polynomial satisfying
\[
f_{1,2}\cdot q\in K[f_{ij}].
\] 
Then $q$ already belongs to the ring $K[f_{ij}]$.
\end{lemma}

\begin{proof}
Write $p=f_{1,2}\cdot q$ uniquely as a linear combination of the standard products
\[
p=\sum_{\alpha\in I} c_{\alpha}F_{\alpha}.
\]
We prove that each $F_{\alpha}$ is divisible by $f_{1,2}.$ 

Notice that $p|_{(x_1,y_1)=(0,0)}= 0$ and also $p|_{(x_2,y_2)=(0,0)}= 0$. Thus, Lemma \ref{5} applies and guarantees that each $F_{\alpha}$ is divisible by some $f_{1,r}$ and some $f_{2,s}$. We claim that one can pick $r=2$. Assume by contradiction that $F_{\alpha_1},\dots,F_{\alpha_k}$ are not divisible by $f_{1,2}$ and write
\begin{align}\label{6}
p=f_{1,2}\cdot P+\sum_{l=1}^{k}c_{\alpha_l}F_{\alpha_l}
\end{align}
for some $P\in K[f_{i,j}]$, a sum of standard products of smaller degree than the degree of $p$, and some coefficients $c_{\alpha_l}\in K$. We enlarge now the polynomial ring $K[\xx,\yy]$ to $K[\xx,\yy,\lambda]$ by adding a new variable $\lambda$ which will be used for weighting the standard products. The goal now is to carry out a suitable substitution such that, whereas the left-hand side of the equation (\ref{6}) and also the summand $f_{1,2}\cdot P$ of the right-hand side vanish, the sum $\sum_{l=1}^{k}c_{\alpha_l}F_{\alpha_l}$ transforms into another sum of standard weighted products. This would yield a contradiction.

Let $\lambda$ be a variable. We perform the substitution $^*:(x_1,y_1)\mapsto (\lambda x_2, \lambda y_2)$ under which $f_{1,k}\mapsto \lambda f_{2,k}$ for all $k\geq 2$ and $f_{i,j}\mapsto f_{i,j}$ if $i,j\neq 1$. Obviously, all standard products (and arbitrary polynomials), that are divisible by $f_{1,2}$, become zero after this substitution. Further, each standard product $F_{\alpha_l}$ with the corresponding diagram
\[
\begin{bmatrix}
1 & \cdots & 1 & 2 & \cdots & 2 & i_{b_l+1} & \cdots \quad\\
j_1 & \cdots & j_{a_l} & j_{a_l+1} & \cdots &  j_{b_l} & i_{b_l+1} & \cdots\quad
\end{bmatrix}
\]
transforms under $^*$ into the standard product $F_{\alpha_l}|_{(x_1,y_1)=(x_2,y_2)}$, with the corresponding diagram 
\[
\begin{bmatrix}
2 & \cdots & 2 & 2 & \cdots & 2 & i_{b_l+1} & \cdots \quad\\
j_1 & \cdots & j_{a_l} & j_{a_l+1} & \cdots & j_{b_l} & i_{b_l+1} & \cdots \quad
\end{bmatrix},
\]
multiplied by $\lambda^{a_l}$, where $a_l$ is the number of columns with $1$ on the top in the diagram of $F_{\alpha_l}$.  It follows that the substitution~$^*$ acts injectively on the standard products that are not divisible by $f_{1,2}$. After applying $^*$, the equality (\ref{6}) becomes
\[
0=p^*=\sum_{l=1}^kc_{\alpha_l}F^*_{\alpha_l}.
\]
Now, since the substitution $^*$ is injective on $F_{\alpha_l}$'s, we conclude $\sum_{l=1}^kc_{\alpha_l}F_{\alpha_l}=0.$
\end{proof}

We are now in a position to prove the First Fundamental Theorem for $\SL_2(K)$.

\begin{proof}[Proof of the First Fundamental Theorem for $\SL_2(K)$]
First observe that, if $p\in K[\xx,\yy]^{\SL_2}$ is an invariant polynomial, then by the linearity of the action it follows that each of its homogeneous summands must be invariant as well. Therefore, we can assume without loss of generality that $p$ itself is a homogeneous polynomial of some degree $k$.

From now on, we will refer to the polynomial $p=p(x_1,y_1,\dots,x_n,y_n) \in K[\xx,\yy]$ as $p(x_i,y_i)$ in order to shorten the notation.
Let us now consider the matrix
\[
S=\begin{pmatrix}
t & 0\\
0 & t^{-1}
\end{pmatrix},
\]
where $t\in K^*$ arbitrary. Then, as $S$ is an $\SL_2(K)$-matrix, $p$ stays invariant under its action:
\begin{align}\label{Saction}
p(x_i,y_i)=S\cdot p(x_i,y_i)=p(tx_i,t^{-1}y_i).
\end{align}
Because $|K|=\infty$ and this equality holds for all $t \neq 0$, the principle of Weyl applies and we can consider $t$ as a variable. Now, comparison of terms of equal degree on the left- and the right-hand sides of $(\ref{Saction})$ shows that each term of $p$ must contain as many variables from the set $\{x_1,\dots,x_n\}$ as from the set $\{y_1,\dots,y_n\}$, if counted with multiplicity. Hence, $p$ is not only homogeneous of an even degree $k=2m$, for some $m$, but it is also homogeneous in the variables $x_1,\dots,x_n$ of degree $m$ and also in the variables $y_1,\dots,y_n$ of the same degree.

Let us consider any invertible matrix $G\in\GL_2(K)$ and examine its action on $p$. We write the matrix in the following way:
\begin{equation*}
   G=G\cdot\begin{pmatrix}\det(G)^{-1} & 0\\ 0 & 1 \end{pmatrix} \cdot\begin{pmatrix}\det(G) & 0\\ 0 & 1 \end{pmatrix}.
\end{equation*}
Obviously, the product 
\[
G\cdot\begin{pmatrix}\det(G)^{-1} & 0\\ 0 & 1\end{pmatrix}
\]
is an element of $\SL_2(K).$ Hence, as $p$ is invariant under the action of $\SL_2(K)$, the action of $G$ on $p$ reduces to the action of 
\[
\tilde{G}\coloneqq\begin{pmatrix}\det(G) & 0\\ 0 & 1 \end{pmatrix}.
\]
We obtain the equality
\begin{align}\label{8}
    p(ax_i+by_i,cx_i+dy_i)&=G\cdot p(x_i,y_i) =\tilde{G}\cdot p(x_i,y_i)=p(\det(G)x_i,y_i)=\nonumber\\
    &= \det(G)^m p(x_i,y_i)=(ad-bc)^mp(x_i,y_i),
\end{align}
using that $p$ has degree equal to $m$ in the variables $x_1,\dots,x_n$. Here, $a,b,c,d \in K$ denote the entries of the matrix $G$. This equality holds on the Zariski open set \[\{(x_1,y_1,\dots,x_n,y_n,a,b,c,d)\in K^{2n+4}:ad-bc\neq 0\}.\] As $|K|=\infty$, Weyl's principle applies and we can consider $a,b,c,d$ as variables. After substituting $a=-y_1, b=x_1, c=-y_2, d=x_2$, the equality (\ref{8}) becomes 
\begin{align}\label{4}
     f_{1,2}^{m}\cdot p(x_i,y_i)=p\left(f_{1,i},f_{2,i}\right) \in K[\xx,\yy].
\end{align}
Now, the claim follows by Lemma \ref{lem:qf12}.
\end{proof} 

Notice that the equality (\ref{4}) also shows that the invariant polynomial $p$ belongs to the intersection $K[\xx,\yy] \cap K(f_{1i},f_{2i})$, where $K(f_{1i},f_{2i}) = K(f_{1,i},f_{2,i}: i=3,\dots,n)$. Thus, one could intuitively think that $p$ does not only lie in the polynomial ring $K[f_{ij}]$ but that it belongs already to the subring $K[f_{1i},f_{2i}] = K[f_{1,i},f_{2,i}: i=3,\dots,n]$. However, the inclusion $K[\xx,\yy] \cap K(f_{1i},f_{2i}) \subseteq K[f_{1i},f_{2i}]$ is wrong. For example, because of the Plücker relation, the polynomial $f_{3,4}$ can be written as
\[
f_{3,4}  = \frac{f_{1,3}f_{2,4} - f_{1,4}f_{2,3}}{f_{1,2}}.
\]
Therefore, $f_{3,4}$ is obviously an element of $K[\xx,\yy]\cap K(f_{1i},f_{2i})$ but there is no reason for it to be contained in $K[f_{1i},f_{2i}]$. In fact, we can easily show the following equality first observed by Igusa in \cite[Thm. 3]{Ig54}:

\begin{corollary}[Igusa]
It holds that $K(f_{ij})\cap K[\xx,\yy]=K[f_{ij}]$.
\end{corollary}
 
\begin{proof}
    The First Fundamental Theorem ensures that $K(f_{ij})\cap K[\xx,\yy] \subseteq K[\xx,\yy]^{\SL_2} = K[f_{ij}]$. The other inclusion is obvious.
\end{proof} 

\begin{remark}\normalfont
Note that the statements of the First Fundamental Theorems for $\GL_2(K)$ and $\SL_2(K)$ are wrong for any finite field $K$. In fact, for any prime power~$q$, the polynomials
\begin{align*}
    p_i \coloneqq x_i^qy_i - x_iy_i^q, \quad  1\leq i\leq n,
\end{align*}
are semi-invariant under the action of $\GL_2(\mathbb{F}_q)$, i.e., we have 
\begin{align}\label{character}
G\cdot p_i = \det(G) p_i.
\end{align}
Therefore, it follows that $p_i \in \mathbb{F}_q[\xx,\yy]^{\SL_2}$
and
$p_i/p_j \in\mathbb{F}_q(\xx,\yy)^{\GL_2}$, for $i\neq j$. However, it is obvious that $p_i \not \in \mathbb{F}_q[f_{ij}]$, showing that $\mathbb{F}_q(f_{ij})  \subsetneq \mathbb{F}_q(\xx,\yy)^{\GL_2}$ and $\mathbb{F}_q[f_{ij}]  \subsetneq \mathbb{F}_q[\xx,\yy]^{\SL_2}$. 
\end{remark}

\medskip\noindent {\bf Acknowledgements.} We are grateful to Michel Brion, Hanspeter Kraft and Herwig Hauser for their valuable discussions and corrections of the first version of this text. The authors also thank Josef Schicho for his great pedagogical introduction to invariant theory at the workshop in Wesenufer, Austria. Finally, we are deeply indebted to Alin Bostan for his time, ideas and comments. \\
The first author was supported by Austrian Academy of Sciences, ÖAW, Doc Stipendium FA506081 and the second author was supported by Austrian Science Fund, FWF, Project AP31338.

\sloppy

\bibliographystyle{alphaabbr}
\bibliography{bib}

\end{document}